\documentclass[12pt]{amsart}

\usepackage{tikz} 

\usepackage{multicol}
\usepackage{color}
\definecolor{red}{rgb}{1.00,0.00,0.00}

\newcommand{\hema}[1]{{\color{blue} \sf $\star\star$ Hema: [#1]}}
\newcommand{\phil}[1]{{\color{red} \sf $\star\star$ Philippe: [#1]}}
\newcommand{\hemadel}[1]{{\color{green} \sf $\star\star$ Hema: [#1]}}
\numberwithin{equation}{section}

\newtheorem{theorem}{Theorem}[section]
\newtheorem*{thm}{Theorem}
\newtheorem{lemma}[theorem]{Lemma}

\newtheorem{proposition}[theorem]{Proposition}
\newtheorem{example}[theorem]{Example}

\newtheorem{remark}[theorem]{Remark}
\newtheorem{definition}[theorem]{Definition}

\newtheorem{question}{Question}

\def\sA{\langle A\rangle}
\def\sAg{\langle A(G)\rangle}

\def\sB{\langle B\rangle}

\def\sC{\langle C\rangle}
\def\sD{\langle D\rangle}

\def\rA{k[A]}
\def\rAg{k[A(G)]}

\def\rB{k[B]}

\def\C{\mathfrak{C}}
\def\N{\mathbb{N}}
\def\Q{\mathbb{Q}}
\def\Z{\mathbb{Z}}

\def\PP{\mathbb{P}}

\def\a{{\bf a}}
\def\b{{\bf b}}

\def\u{{\bf u}}
\def\t{{\bf t}}

\def\rk#1{\hbox{\rm rank}\,(#1)}

\def\height#1{\hbox{\rm ht}\,(#1)}

\begin{document}
\title{Gluing and splitting of homegeneous toric ideals}

\author{Philippe Gimenez}
 \address{IMUVA-Mathematics Research Institute, Universidad de Valladolid, 47011 Valladolid, Spain.}
 \email{pgimenez@uva.es}
 
\author{Hema Srinivasan}
 \address{Mathematics Department, University of Missouri, Columbia, MO 65211, USA.}
 \email{SrinivasanH@missouri.edu}

\thanks{The first author was partially supported by grant PID2022-137283NB-C22 funded by MCIN/AEI/ 10.13039/501100011033 and by ERDF "A way of making Europe".
The second author was partially supported by a grant from Simons Foundation.
}
\maketitle  
\begin{abstract}
We show that any two homogeneous affine semigroups can be glued by embedding them suitably in a higher dimensional space.
As a consequence, we show that the sum of their homogeneous toric ideals is again a homogeneous toric ideal, and that the minimal graded free resolution of the associated semigroup ring is the tensor product of the minimal resolutions of the two smaller parts.  We apply our results to toric ideals associated to graphs to show how two of them can be a splitting of a toric ideal associated to a graph or an hypergraph. 
\end{abstract}

\section*{Introduction}

Let $A$ be an $n\times p$ matrix over the integers.  The toric ideal $I_A$ of $A$ is a binomial ideal in $R=k[x_1, \ldots x_p]$ generated by the binomials 
$x^{\alpha_+} - x^{\alpha_-}$ where $\alpha=\alpha_+- \alpha_-\in \Z^p$ satisfies $A\cdot\alpha= 0$. The ideal $I_A$ is necessarily prime and the dimension of $R/I_A$ is precisely the rank of the matrix $A$.  
We say an $n\times p$ matrix over the integers is homogeneous if the  toric ideal $I_A$ is homogeneous with the standard grading.  

\smallskip

Gluing is a construction by which a semigroup $\sA$ of embedding dimension $p$ is joined with a semigroup $\sB$ of embedding dimension $q$ to produce a semigroup $\sC$ of embedding dimension $p+q$ with a strong requirement on 
the three associated toric ideals.
This procedure is powerful in that $\sC$ is complete intersection, Cohen-Macaulay or Gorenstein if and only if $\sA$ and $\sB$ are too.  Further, it is possible to compute explicitly all homological invariants of $\sC$ from those of $\sA$ and $\sB$ for the resolution of $\sC$ is explicitly obtained from that of $\sA$ and $\sB$. The notion of gluing in the context of numerical semigroups was introduced by Delorme \cite{De} 
and further generalized, studied and analyzed by many others;
see, e.g., \cite{rosales}, \cite{FMS}, \cite{RG}.  

\smallskip

Two numerical semigroups can always be glued.  However, subsemigroups of $\N^n$ for $n\ge 2$, may not always be glued.  Indeed, by the necessary condition for gluing in \cite{Res22}, two subsemigroups of $\N^n$ given by matrices $A$ and $B$ can be glued into a subsemigroup of $\N^n$ only if $\rk A+ \rk B -1 \leq n$.  Thus, if $A$ and $B$ have both maximal rank $n$ and $n\geq 2$, then they cannot be glued in $\N^n$.  This raises the question that given $\sA$ and $\sB$ in $\N^n$ and $\N^m$ respectively, can we embed them in a higher dimensional space, so that they can be glued in $\N^t$?  We answer this question in the affirmative for homogeneous semigroups by giving explicit constructions. As a consequence, we obtain the following result: 

%

\begin{thm}[Theorem \ref{thm:A+B=C}]
Let $I_A\subset k[x_1, \ldots, x_p]$ and $I_B\subset k[y_1, y_2, \ldots y_q]$ be two homogeneous toric ideals associated to two $n\times p$ and $m\times q$ matrices  $A$ and $B$ respectively. Set $R:=k[x_1, \ldots, x_{p},z, y_1, \ldots, y_q]$.  For any $1\le i\le p, 1\le j\le q$,  consider the ideals $I_A':=I_A\vert_{x_i=z}\cdot R$ and $I_B':=I_B\vert_{y_j=z}\cdot R$.
Then, 
$I_C:=I_A'+I_B'$ is a homogeneous toric ideal, and $R/I_C$ is isomorphic to the semigroup ring of a suitable $(m+n-1)\times (p+q-1)$ matrix $C$. 
Moreover, if  $F_A$ and $F_B$ are minimal graded free resolutions of $R/I_A'$ and $R/I_B'$ respectively, then $F_A\otimes _k F_B$ is a minimal graded resolution of $k[C]$. 
\end{thm}

Using the theorem above, we can recover, explain and add to the existing notions on splitting of graphs and toric ideals associated to them.  A graph $G$ is said to split along an edge $e$ into two subgraphs $G_1$ and $G_2$ if $G_1$ and $G_2$ have exactly one edge $e$ (and its two vertices) in common, and their union is $G$. We will say that the corresponding toric ideal $I_G$ splits if, moreover, $I_G = I_{G_1}+I_{G_2}$. 
We show that $I_G$ is a splitting of $I_{G_1}$ and $I_{G_2}$ if and only if at least one of $G_1$ or $G_2$ is bipartite. 
This recovers a result in \cite{FHKV}.  
%
Further, we will show that any two graphs $G_1$ and $G_2$ are a splitting of a 3-uniform hypergraph in the following sense: we can embed the incidence matrices of $G_1$ and $G_2$ in larger matrices that are homogeneous of degree $3$, and glue them to get the incidence matrix of a 3-uniform hypergraph $G$ such that $I_G = I_{G_1}+I_{G_2}$.   

\smallskip

The paper is organized as follows.  In the first section, we define homogeneous toric ideals and explain how the gluing is achieved.
In section \ref{sec:splitting}, splitting of toric ideals is defined with an application of homogeneous gluing to homogenous and non-homogneous toric ideals. An application and an interpretation of this concept to graphs and hypergraphs is in section \ref{sec:graphs}, followed by illustrating examples in section \ref{sec:examples}. 

\section{Gluing two homogeneous semigroups}\label{sec:gluing}

In this first section, we provide a completely general way of gluing two homogeneous semigroups by using the matrices that define them.  
To begin with, we will take the entries in $A$ to be non negative integers. We will see later in Remark \ref{rk:nonnegative} that, in the homogeneous case, we can always assume this without loss of generality.

\smallskip
Let $A= (a_{ij})$ be an $n\times p$ matrix with entries in $\N$. We can consider the columns $\a_j$ of $A$ as elements of $\N^n$ and denote also by $A$ the set $\{\a_1,\ldots,\a_p\}$. 
%
The subsemigroup of $\N^n$ generated by the columns of  $A$ will be denoted by $\sA$.  Let $k$ be an arbitrary field. The semigroup ring of $A$ is $\rA \simeq k[x_1, \ldots, x_p]/I_A$, where $I_A$ is the kernel of the map $\phi_A: k[x_1,\ldots, x_p]\to k[t_1, \ldots, t_n]$ defined by $\phi_A(x_j) = \prod_{i=1}^n t_i^{a_{ij}}$. Denote by $\t^{\a_j}$ the monomial $\prod_{i=1}^n t_i^{a_{ij}}$. The ideal $I_A$ is a binomial ideal generated by $x^{\alpha_+} - x^{\alpha_-}$ with $\alpha_+,\alpha_-\in\N^p$ having disjoint supports and where $\alpha=\alpha_+- \alpha_-\in \Z^p$ satisfies $A\cdot\alpha= 0$. This $I_A$ is called a toric ideal. It is, indeed, a prime ideal of height $p-r$ where $r$ is the rank of the matrix $A$; see, e.g., \cite[Lemma 4.2]{sturm}. 

\begin{definition}{\rm Two matrices $A$ and $B$ over natural numbers are  {\it equivalent} if their toric rings  $K[A]$ and $k[B]$ are equal.  We write $A \sim  B$.  Indeed, this $\sim $ is an equivalence relation on the set of matrices over natural numbers. 
}\end{definition}

\medskip
The matrix $A$
is said to be {\it homogeneous} if there exist $\lambda_1, \ldots, \lambda_n\in \Z$ and $d>0$ such that $(\lambda_1 \cdots \lambda _n)\times A= d\cdot (1\cdots 1)$. It should be noted that neither the $\lambda_i$'s nor $d$ are unique. Still, when this occurs, we will sometimes say that $A$ is homogeneous of degree $d$ (knowing that a homogeneous matrix may have several distinct degrees).  Equivalently, $A$ is homogeneous if there are $\lambda_1, \ldots, \lambda_n\in \Q$ such that $\sum_{i=1}^n \lambda_i a_{ij} = 1, 1\le j\le p$.  Here the $\lambda_i$'s may still be not unique.

\begin{example}{\rm
For $A = \begin{pmatrix} 
0&1&2&3 \\ 3&2&1&0 \\ 1&1&1&1
\end{pmatrix}
$, one has that $\lambda_1=\lambda_2=1$, $\lambda_3=0$ and $d=4$ works but also $\lambda_1=\lambda_2=0$, $\lambda_3=1$ and $d=1$.
}\end{example}

For an element $\alpha \in \N^p$, denote by $|\alpha|$ the sum of the entries in $\alpha$. 
When $A$ is homogeneous, $A\cdot\alpha = 0 \implies (1\cdots 1)\cdot \alpha = 0$, and hence   $|\alpha| = 0$ and conversely. Thus, the matrix $A$ is homogeneous if and only if the toric ideal $I_A$ is homogeneous and the semigroup ring $\rA$ is graded with the standard grading; see, e.g., \cite[Lem. 4.14]{sturm}.

\smallskip

It is easy to see that some row operations on a homogenous matrix do not change the associated toric ideal and we collect them in the following lemma for ease in reference.
\begin{lemma}\label{lem:homogeneousMatrix}
Let $A$ be a homogeneous $n\times p$ matrix with $(\lambda_1 \cdots \lambda _n)\times A= d\cdot (1\cdots 1)$ for some $\lambda_1, \ldots, \lambda_n\in \Z$ and $d>0$.
Then, the following operations on the rows of $A$ do not change the associated toric ideal and the semigroup ring, and hence, the result matrix $A'$ obtained by applying operations of this type to $A$ is equivalent to the original matrix $A$, i.e., $A' \sim A$:
\begin{enumerate}
\item multiply all the entries of a row  by the same non-zero integer, simplify them by a common factor, or add a multiple of one row to another (this holds also in the non-homogeneous case),
\item add to all the entries of a row $i$, $1\leq i\leq n$, the same integer $s_i$ provided $d+\lambda _is_i\neq 0$,
\item add to the matrix $A$ a new row where all the entries are identical.  
\end{enumerate}
\end{lemma}

Note that a matrix $A'$ obtained by applying Lemma \ref{lem:homogeneousMatrix} to $A$ may not have the same number of rows as $A$, but $A$ and $A'$ have the same number of columns and the same rank. 

\begin{remark}\label{rk:nonnegative}{\rm
Toric ideals are sometimes defined as we did at the beginning of this section but considering a matrix $A$ with entries in $\Z$. In this case, the map $\phi_A$ has to be considered from $k[x_1,\ldots,x_p]$ to the Laurent polynomial ring $k[t_1^\pm,\ldots,t_n^\pm]$.
Note that if $A$ is a homogeneous matrix in $\Z$, with $\lambda \in \Z^n$ such that $\lambda ^t A = d(1, \ldots, 1)$ and $d>0$, there exists $z\in \N$ such that $a_{ij}+ z \in \N$ for all $i, j$ and $d+z(\sum_{i=1}^n \lambda_i)\neq 0$.  Indeed, if $d+z(\sum_{i=1}^n \lambda_i)= 0$, we can take $z+1$ as $z$ does not have to be the smallest so as to make the entries in $A$ non negative. Since by Lemma \ref{lem:homogeneousMatrix}, adding $z$ to all the entries in $A$ will not change the toric ideal $I_A$, we can assume that our homogeneous matrix is over the natural numbers without loss of generality and we will do so when convenient.  That is,  if $A$ is a homogeneous matrix over $\Z$, then there is a homogneous matrix $A'$ over $\N$ such that $A'\sim A$.
}\end{remark}

Consider now $A = (a_{ij})$ an $n\times p$ homogeneous matrix of rank $r$ and degree $\deg A$, and  $B = (b_{ij})$ an $m\times q$ homogeneous matrix of rank $s$ and degree $\deg B$.  The semigroup $\sA$ generated by the columns of $A$ is in $\N^n$ and  the semigroup $\sB$ generated by the columns of $B$ is in $\N^m$. The semigroup rings of $A$ and $B$ are, respectively, $\rA \simeq k[x_1, \ldots, x_p]/I_A$ and $\rB\simeq  k[y_1, \ldots, y_q]/I_B$.   Let $\lambda_i $ and $\mu_i$ be such that $\sum_{i=1}^n\lambda_ia_{ij} = \deg A $ and $\sum_{i=1}^m \mu_ib_{ij} = \deg B$.
Choose any column in $A$ and any column in $B$.  
By rearranging the $x_i$'s and $y_j$'s, we can take them to be the last column of $A$ and the first column of $B$. Moreover, by rearranging the rows of the matrices (which does not change the semigroup rings since it corresponds to relabeling the parameters $t$'s), one can always assume that $a_{np}\neq 0$ and $b_{11}\neq 0$ as well as the multiples, $\lambda_n$ of the last row of $A$ and $\mu_1$ of the first row of $B$, are positive. This assures that $\deg A+\lambda _n t $ and $\deg B+\mu_1 t$ will not be zero for any $t\neq 0$.   
Now, by Lemma \ref{lem:homogeneousMatrix}, we can add any number to the whole row of a matrix without changing the corresponding ideal and the semigroup ring. Set $e: = a_{np} -b_{11} $ and $\delta := -1$ if $e\le 0$  and $0$ otherwise. It is easy to check that $a_{np}+\delta e=b_{11}+(1+\delta)e$.

\smallskip

Consider now the following two matrices $A'$ and $B'$, both with $n+m-1$ rows, and $p$ and $q$ columns respectively:
\[
A' = \begin{pmatrix} 
a_{11} & a_{12}& \ldots & a_{1,p-1}  & a_{1p} \\
\vdots &\vdots&\ldots & \vdots& \vdots \\
a_{n-1,1} & a_{n-1,2}  &\ldots &a_{n-1,p-1} &a_{n-1,p} \\
a_{n1}+\delta e & a_{n2}+\delta e  &\ldots &a_{n,p-1}+\delta e &a_{np}+\delta e\\
b_{21}&b_{21} &\ldots&b_{21} &b_{21} \\
\vdots &\vdots&\ldots & \vdots& \vdots \\
b_{m1}  &b_{m1}&\ldots&b_{m1} &b_{m1}
\end{pmatrix},
\]
\[
B' = \begin{pmatrix} a_{1p}&a_{1p}&\ldots&a_{1p}&a_{1p}\\
\vdots &\vdots&\ldots & \vdots& \vdots \\
a_{n-1,p} &a_{n-1,p}&\ldots&a_{n-1,p}&a_{n-1,p}\\
b_{11}+(1+\delta) e &b_{12}+(1+\delta) e&\ldots&b_{1,q-1}+(1+\delta) e&b_{1q}+(1+\delta) e\\
b_{21}& b_{22}& \ldots& b_{2,(q-1)}&b_{2q}\\
\vdots &\vdots&\ldots & \vdots& \vdots \\
b_{m1} &b_{m2}&\ldots&b_{m,q-1}&b_{mq}
\end{pmatrix},
\]
and let $\tilde{C}$ be the $(n+m-1)\times (p+q)$ matrix obtained by concatenating both matrices, $\tilde{C}=(A'\vert B')$.

\begin{proposition}\label{prop:rankC}
    $\rk{\tilde{C}}=\rk{A}+\rk{B}-1$.
\end{proposition}

\begin{proof} 
Set $r:=\rk{A}$ and $s:=\rk{B}$.
The matrix $A$ is homogeneous and hence any $X$ such that $AX = 0$ satisfies $\sum_{i=1}^p x_i = 0$. Thus, $A'X= 0$ $\Longleftrightarrow$ $AX= 0$ and 
$\rk{A} =\rk{A'}$.  Similarly,  we have $\rk{B}=\rk{B'}$. The set of solutions to $A'X= 0$, which is the same as the set of solutions to $AX= 0$, is a $\Q$-vector space of dimension $p-r$ with basis say $v_1, \ldots , v_{p-r}$, and let $w_1, \ldots, w_{q-s}$ be a basis for the $\Q$-vector space of solutions to $B'Y = 0$ (or $BX=0$).  
  
Let us show that $T=\{
\overline{v_i}=\begin{pmatrix}
v_i\\
0_{q\times 1}\\
\end{pmatrix},
\overline{w_j}=\begin{pmatrix} 
0_{p\times 1}\\
w_j\\
\end{pmatrix},
u=\begin{pmatrix}
0_{p-1\times 1}\\
1\\
-1\\
0_{q-1\times 1}
\end{pmatrix}\} $
is a linearly independent set of solutions to $CX= 0$.
Elements in $T$ are linearly independent because the $v_i$'s and $w_j$'s are independent of each other as they have non zero entries in complementary rows. 
Moreover, if 
\[\alpha_1\overline{v_1}+\cdots+\alpha_{p-r}\overline{v_{p-r}}
+
\gamma_1\overline{w_1}+\cdots+\gamma_{q-s}\overline{w_{q-s}}
= \beta u,
\]
then $\alpha_1v_1+\cdots+\alpha_{p-r}v_{p-r}
=\beta \begin{pmatrix}
0\\ \vdots\\0\\ 1
\end{pmatrix} $, and 
$\gamma_1 w_1+\cdots+\gamma_{q-s}w_{q-s}=\beta \begin{pmatrix}
-1\\ 0\\ \vdots\\0
\end{pmatrix} $.
But this is not possible since the entries in $v_i$ (or $w_j$) all add up to zero and hence no linear combination of $v_i$ (or  $w_j$'s) will have only one nonzero term.
Thus,
$\rk{\tilde{C}}\le r+s-1$. 

\smallskip
Now let $C$ be the $(m+n-1)\times (p+q-1)$ submatrix of $\tilde{C}$ obtained by deleting the $p^{th}$ column of $\tilde{C}$ which is the same as its $p+1^{st}$ column.   
The rank of $C$ is the same as that of $\tilde{C}$.  We will show that the rank of $C$ is $r+s-1$ by showing that the $p-r+q-s$ vectors in $T'= \{
\tilde{v_i}=\begin{pmatrix}
v_i\\
0_{(q-1)\times 1}\\
\end{pmatrix},
\tilde{w_j}=\begin{pmatrix} 
0_{(p-1)\times 1}\\
w_j\\
\end{pmatrix}\}$  generate the set of solutions to $CX=0$.  We will proceed with the case $b_{11}\ge a_{np}$, so that $\delta =-1$ (the case $\delta =0$ is similar). 
%
Consider $z= \begin{pmatrix} x_1\\
\vdots\\
x_{p-1}\\
y_1\\
\vdots \\
y_q\\
\end{pmatrix}$ a solution to $CX=0$.  Then setting $x:= \sum_{j=1}^{p-1} x_j$ and $y:= \sum_{j=1}^qy_j$, we get the following $m+n-2$ relations:
\begin{equation}\label{eq:A}
\sum_{j=1}^{p-1} a_{ij}x_j+a_{ip}y = 0,\ \forall i, 1\le i\le n-1\,,
\end{equation} 
\begin{equation}\label{eq:B}
b_{i1}(x+y_1)+\sum_{j=2}^qb_{ij}y_j = 0,\ \forall i, 2\le i\le m\,.
\end{equation}
Note that one has one more relation, 
$\sum_{j=1}^{p-1} (a_{nj}+\delta e)x_j+ \sum_{j=1}^qb_{1j}y_j=0$.
%
%
Further, since $C$ is homogeneous, $x+y = 0$. 
 
\smallskip
Consider $\lambda_1,\ldots,\lambda_n$ and $d=\deg A$ such that $\sum_{i=1}^n\lambda_i a_{ij} = d$ for all $j=1,\ldots,p$, with $\lambda_n >0$. Then, we get from \eqref {eq:A},
\[\sum_{j=1}^{p-1} (d-\lambda_na_{nj})x_j+(d-\lambda_n a_{np})y   = 0.   \]\
So, $ \lambda_n (\sum_{j=1}^{p-1}a_{nj}x_j +a_{np}y) = d(x+y) = 0$ because $x+y=0$, and  
$\sum_{j=1}^{p-1}a_{nj}x_j-a_{np}x= 0$.  
%
Thus, for some $c_1,\ldots,c_{p-r}$,
$$\begin{pmatrix} x_1\\
\vdots\\
x_{p-1}\\
-x\\
0\\
\vdots \\
0\\ 
\end{pmatrix} =\sum_{j=1}^{p-r}c_j\tilde{v_j}\,.$$

Similarly, if $\mu_1,\ldots,\mu_m$ and $f=\deg B$ are such that $\sum_{i=1}^m \mu_ib_{ij}= f$ for all $j=1,\ldots,q$, and with $\mu_1>0$, one gets from \eqref{eq:B} that
$\displaystyle{\sum_{j=2}^{q} (f-\mu_1 b_{1j})y_j+(f-\mu_1 b_{11})(x+y_1 )  = 0}$. Hence
$\mu_1(\sum _{j=2}^q  b_{1j}y_j+ b_{11}(x+y_1))=f(x+y) =0$, and so
$\sum _{j=2}^q  b_{1j}y_j+ b_{11}(x+y_1)=0$. Thus, for some $d_1,\ldots,d_{q-s}$, 
$$
\begin{pmatrix} 0\\
\vdots\\
0\\
x+y_1\\
\vdots \\
y_q\\ 
\end{pmatrix} =\sum_{i=1}^{q-s}d_i\tilde{w_i}\,,$$
and $z= \sum_{j=1}^{p-r}c_j\tilde{v_j}+\sum_{i=1}^{q-s}d_i\tilde{w_i}$.
This proves, as desired, that $T'$ generates the set of solutions to $CX=0$, and hence
the rank of $C$, i.e., the rank of $\tilde{C}$, is $p+q-1- (p-r +q-s)= r+s-1$. 
\end{proof}

\begin{theorem}\label{thm:gluing} If $A$ and $B$ are homogeneous $n\times p$ and $m\times q$ matrices,
we can embed $A$ and $B$ in homogeneous matrices $A'$ and $B'$ without changing their toric ideals, 
i.e., there exists $A'$ and $B'$ with $A\sim A'$ and $B\sim B'$,
such that for $\tilde{C}=(A'\vert B')$, $\langle \tilde{C}\rangle$ is a gluing of $\sA$ and $\sB$.
Indeed, one can choose $A'$ and $B'$ such that, for some $i$ and $j$, $1\leq i\leq p$, $1\leq j\leq q$, 
$$I_{\tilde{C}}=I_{A'}+I_{B'}+\langle x_i-y_j\rangle  = I_A+I_B +\langle x_i-y_j\rangle\,.$$ 
\end{theorem}

\begin{proof}
As already mentioned, one can assume without loss of generality that $i=p$ and $j=1$.
By the proof of Proposition \ref{prop:rankC} above, we see that any solution to $\tilde{C}X= 0$ is a linear combination of the vectors in $T$, i.e., if $\tilde{C}z= 0$, then 
$z = \sum_{i=1}^{p-r}r_i\overline{v_i}+ \sum_{i=1}^{q-s} s_i \overline{w_i}+ \alpha u$ 
for some $r_i, s_i, \alpha$. Thus, $I_{\tilde{C}}= I_A'+I_B' + \langle x_p-y_1\rangle$ and, by definition, $\langle \tilde{C}\rangle$ is a gluing of $\sA$ and $\sB$. 
\end{proof}

\begin{remark}{\rm
In the proof of above Theorem \ref{thm:gluing}, we did not use the sufficient condition for the existence of gluing \cite[Thm. 2.8]{Res22}.  Indeed, if the entries in the $p$th column of $A'$ and the entries in the first column of $B'$ are relatively prime, then we can have an alternate proof as follows:  since the last column $\a_p$ of $A'$ is the same as the first column of $B'$, we have $\a_p $ is the gluable lattice point defined in \cite[Def 2.1] {Res22} and it is in $\langle A'\rangle \cap \langle B'\rangle$.  Hence by Proposition \ref{prop:rankC} and \cite[Thm. 2.8]{Res22}, for $\tilde{C}=(A' \mid B')$, $\langle \tilde{C}\rangle$ is a gluing of $\sA$ and $\sB$, and the gluing binomial is $x_p-y_1$.  
}\end{remark}

\begin{example}\label{ex:gluing3x5}{\rm 
Consider the matrices
$A=\begin{pmatrix} 2&0&6&4&0\\ 3&4&1&0&2\\ 2&3&0&3&5\end{pmatrix}$ and
$B=\begin{pmatrix} 3&2&2&5&0\\ 3&2&4&0&1\\ 0&2&0&1&5\end{pmatrix}$. Then, $I_A=\langle x_1x_4-x_3x_5,  x_2^2x_4-x_1^2x_5, x_1^3-x_2^2x_3\rangle\subset R_A:=k[x_1,\ldots,x_5]$ and $I_B=\langle y_1y_2^3-y_3^2y_4y_5, y_2^2y_3^3y_4-y_1^5y_5, y_2^5y_3-y_1^4y_5^2,  y_1^6y_2-y_3^5y_4^2, y_2^8-y_1^3y_3y_4y_5^3\rangle\subset R_B:=k[y_1,\ldots,y_5]$. The matrix built in Proposition \ref{prop:rankC} and Theorem \ref{thm:gluing} is
\[\langle \tilde{C}\rangle=\begin{pmatrix} 
2&0&6&4&0&0&0&0&0&0\\ 3&4&1&0&2&2&2&2&2&2\\ 2&3&0&3&5&5&4&4&7&2\\ 
3&3&3&3&3&3&2&4&0&1\\ 0&0&0&0&0&0&2&0&1&5\end{pmatrix}\]
and one can check using, e.g., {\tt Singular} \cite{Sing}, that for $R:=k[x_1,\ldots,x_5,y_1,\ldots,y_5]$, one has $I_{\tilde{C}}=I_A\cdot R+I_B\cdot R+\langle x_5-y_1\rangle\subset R$ as expected by Theorem \ref{thm:gluing}.
This example continues later in Example \ref{ex:splitting3x5}.
}\end{example}

\section{Splitting a homogeneous toric ideal}\label{sec:splitting}

\subsection{The general homogeneous case}
A toric ideal is a binomial prime ideal in a polynomial ring.  Such an ideal is homogeneous if it is homogenous with the standard grading. In general, the sum of two toric ideals may not be toric. When it is, we will call it a splitting.

\begin{definition}
A toric ideal $I$ is a {\it splitting} of two toric ideals $I_1$ and $I_2$, if $I = I_1+I_2$.
\end{definition}

By applying Theorem \ref{thm:gluing}, given two homogeneous toric ideals $I_A$ and $I_B$, the semigroups $\sA$ and $\sB$ can be glued. We will now show that this provides a splitting of a toric ideal $I_C=I_A'+I_B'$ where $I_A'$ and $I_B'$ are copies of $I_A$ and $I_B$ respectively. Moreover, we will have $\height {I_{C}} = \height {I_A}+ \height {I_B}$, and the minimal graded free resolution of the semigroup ring $k[C]$ will be the tensor product of the minimal resolutions of the two smaller parts (that are copies of the minimal graded free resolutions of $k[A]$ and $k[B]$). 

\begin{theorem}\label{thm:A+B=C}
Let $I_A\subset k[x_1, \ldots, x_p]$ and $I_B\subset k[y_1, y_2, \ldots y_q]$ be two homogeneous toric ideals associated to $n\times p$ and $m\times q$ matrices  $A$ and $B$. Set $R:=k[x_1, \ldots, x_{p-1},z, y_2, \ldots, y_q]$.
If $I_A'\subset R$ and $I_B'\subset R$ are the ideals obtained by extending to $R$ the ideals obtained after making $x_p=z$ and $y_1=z$ in $I_A$ and $I_B$ respectively, then $I_{C}:=I_A'+I_B'$ is a homogeneous toric ideal, and $R/I_{C}$ is isomorphic to the semigroup ring of a suitable $(m+n-1)\times (p+q-1)$ homogeneous matrix $C$. 
In particular, the toric ideal $I_{C}$ is a splitting of two toric ideals $I'_A$ and $I'_B$.  

Moreover, if  $F_A$ and $F_B$ are minimal graded free resolutions of $R/I_A'$ and $R/I_B'$ respectively, then $F_A\otimes F_B$ is a minimal graded resolution of $k[C]$.  

Indeed, we can identify any of the $x_i, 1\le i\le p$ with any of the $y_j, 1\le j \le q$ and the resulting sum is the appropriate toric ideal. 
\end{theorem}


\begin{proof}
Existence of a $\tilde{C}$ with the property, $I_{\tilde{C}} = I_{A'}+I_{B'}+(x_p-y_1)=I_{A}+I_{B}+(x_p-y_1)$ follows from Theorem \ref{thm:gluing}. Removing one of the two identical columns in $\tilde{C}$ corresponding to variables $x_p$ and $y_1$, one gets a $(m+n-1)\times (p+q-1)$ matrix $C$.

\smallskip
Now, $x_p-y_1$ is a non zero divisor for $k[X,Y]/I_A+I_B$.   Hence 
$$Tor^{k[X,Y]}_i (k[X,Y]/I_A+I_B, k[X,Y]/(x_p-y_1)) = 0, i>0\,.$$  
Hence $F_A\otimes_kF_B \otimes_{k[X,Y]} k[X,Y]/(x_p-y_1) $ is exact and resolves $R/I_A+I_B+(x_p-y_1) \simeq k[\tilde{C}]$ over $k[X,Y]/ (x_p-y_1)\simeq R$. Now, $I_A+I_B+(x_p-y_1)/(x_p-y_1) \simeq I_A'+I_B' \subset R$.

\smallskip
Hence the ideal $I_A'+I_B'$ is homogeneous toric ideal of codimension $r+s-1$, where $r= \rk A, s= \rk B$.   In otherwords, the complex map from $F_A\otimes F_B\to F_A\otimes F_B$, given by multiplication by $x_p-y_1$ becomes zero map when you first mod out by $x_p-y_1$ and hence the mapping cylinder is indeed $F_A\otimes F_B$.  
\end{proof}

\begin{remark}\label{rk:height}{\rm
In Theorem \ref{thm:A+B=C}, if $A$ and $B$ are two homogeneous matrices of rank $r$ and $s$ respectively, so that $I_A'$ and $I_B'$ are homogeneous toric ideals of height $p-r$ and $q-s$ in $k[x_1, \ldots, x_{p-1},z]$ and $k[z, y_2, \ldots, y_q]$ respectively, then the matrix $C$ will have rank $r+s-1$, and hence
$$\height{I_C} = \height{I_A}+ \height{I_B}\,.$$
}\end{remark}

\begin{example}\label{ex:splitting3x5}{\rm 
Considering the matrices $A$ and $B$ in Example \ref{ex:gluing3x5}, one has that 
in $R=k[x_1,\ldots,x_4,z,y_2,\ldots,y_5]$, for the toric ideals $I_A'=\langle x_1x_4-x_3z,  x_2^2x_4-x_1^2z, x_1^3-x_2^2x_3\rangle$ and $I_B'=\langle zy_2^3-y_3^2y_4y_5, y_2^2y_3^3y_4-z^5y_5, y_2^5y_3-z^4y_5^2,  z^6y_2-y_3^5y_4^2, y_2^8-z^3y_3y_4y_5^3\rangle$ obtained by making $x_5=z$ in $I_A$ and $y_1=z$ in $I_B$ respectively, one has the splitting $I_{C}=I_A'+I_B'$ for the matrix 
\[C=\begin{pmatrix} 
2&0&6&4&0&0&0&0&0\\ 3&4&1&0&2&2&2&2&2\\ 2&3&0&3&5&4&4&7&2\\ 
3&3&3&3&3&2&4&0&1\\ 0&0&0&0&0&2&0&1&5\end{pmatrix}\,.\]
\vskip .2truein
The Betti diagrams of $R_A/I_A$ and $R_B/I_B$ are, respectively, as follows:
\vskip .2truein
\begin{multicols}{2}
{\footnotesize\begin{verbatim}
           0     1     2                                                             
------------------------                                                             
    0:     1     -     -                                                             
    1:     -     1     -                                                             
    2:     -     2     2                                                             
------------------------                                                             
total:     1     3     2 
\end{verbatim}}
\columnbreak
{\footnotesize\begin{verbatim}
           0     1     2     3                                                       
------------------------------                                                       
    0:     1     -     -     -                                                       
    1:     -     -     -     -                                                       
    2:     -     -     -     -                                                       
    3:     -     1     -     -                                                       
    4:     -     -     -     -                                                       
    5:     -     2     1     -                                                       
    6:     -     1     1     -                                                       
    7:     -     1     4     2                                                       
------------------------------                                                       
total:     1     5     6     2  
\end{verbatim}}
\end{multicols}
\smallskip
Since the minimal graded free resolution of $R/I_{C}$ is the tensor product of the minimal graded free resolutions of $k[x_1,\ldots,x_4,z]/I_A'$ and $k[z,y_2,\ldots,y_5]/I_B'$, the Betti diagram of $R/I_{C}$ is
\newpage
\begin{center}
{\footnotesize\begin{verbatim}
           0     1     2     3     4     5                                           
------------------------------------------                                          
    0:     1     -     -     -     -     -                                           
    1:     -     1     -     -     -     -                                           
    2:     -     2     2     -     -     -                                           
    3:     -     1     -     -     -     -                                           
    4:     -     -     1     -     -     -                                           
    5:     -     2     3     2     -     -                                           
    6:     -     1     3     1     -     -                                           
    7:     -     1     9     9     2     -                                           
    8:     -     -     3     8     4     -                                           
    9:     -     -     2    10    12     4                                           
------------------------------------------                                           
total:     1     8    23    30    18     4 
\end{verbatim}}
\end{center}

}\end{example}

\subsection{The two dimensional case}\label{subsec:2dim}
 
In the case of two homogeneous subsemigroups $\sA$ and $\sB$ of $\N^2$, there is a simpler general procedure for gluing them in $\N^3$ and then obtain the corresponding splitting. 

\smallskip
Suppose $R$ is a homogeneous affine semigroup rings of dimension $2$.  Then, without loss of generality, there is a $2\times p$ matrix $A$ over the integers (and indeed over $\N$, by our previous argument) such that $R = k[A]$. 

 \smallskip
Since $A$ is homogeneous, there exists $\lambda_1,\lambda_2\in\Z$ and $c>0$ such that $(\lambda_1 \lambda_2)\cdot A= c\cdot (1, \ldots 1)$. Note that one can assume without loss of generality that $\gcd (\lambda_1, \lambda_2) =1$ and, by changing the rows if necessary, we may choose $\lambda_1 \neq 0$. Moreover, by Lemma \ref{lem:homogeneousMatrix}, one can assume that the rows of $A$ are relatively prime. 

\smallskip
If $\lambda_2=0$, then all the entries on the first row of $A$ are equal, and hence $A$ is equivalent to a matrix whose entries on the first row are all equal to 1. Moreover, if the least nonnegative integer on the second row is non zero, one can substract it to all the entries on the row and hence one can assume that the first entry on the second row is 0.
 
\smallskip
Now if $\lambda_2\neq 0$, i.e., $\lambda_1$ and $\lambda_2$ are both nonzero, we can multiply the first row of $A$ by $\lambda_1$ and then add $\lambda_2a_{21}$ to the first row with without changing the ring $R$. 
Finally, we can also add $-a_{21}$ to the second row (assuming, after eventually reordering the columns, that $a_{21}$ is the smallest entry on the second row) and then multiply by $\lambda_2$ without changing $R$. The new matrix $A'$, equivalent to $A$, will now look like
\[
A' = \begin{pmatrix} 
c & c-a_1& \ldots & c-a_{p-2} & c-a_{p-1}\\
0& a_1  &\ldots &a_{p-2} &a_{p-1}\\
\end{pmatrix}.
\]
Adding the second row to the first one and then simplify by $c$, again one gets an equivalent matrix with all the entries on its first row equal to 1 and the first entry on the second equal to 0.

\begin{remark}{\rm 
    We have shown that associated to an affine semigroup ring $R$ of dimension 2, there is always a numerical semigroup $\langle a_1,\ldots,a_{p-1}\rangle\subset\N$ such that $R=k[A]$ for $A = \begin{pmatrix} 
1 & 1& \ldots & 1 & 1\\
0& a_1 &\ldots &a_{p-2} &a_{p-1}\\
\end{pmatrix}$.
}\end{remark}

Thus, for any two homogeneous affine semigroup rings $R_1$ and $R_2$ of dimension $2$, we can assume without loss of generality that $R_1 = k[A]\simeq k[x_1,\ldots,x_p]/I_A$ and $R_2 = k[B]\simeq k[y_1,\ldots,y_q]/I_B$, where 
\[
A = \begin{pmatrix} 
1 & 1& \ldots & 1 & 1\\
0& a_1 &\ldots &a_{p-2} &a_{p-1}\\
\end{pmatrix}\quad\hbox{and}\quad 
B = \begin{pmatrix}  
1& 1& \ldots& 1&1\\
0& b_1 &\ldots &b_{q-2}&b_{q-1}\\
\end{pmatrix}.
\]
In this case, we glue them easily in $\N^3$ through the following $3\times (p+q-1)$ matrix:
\[
C=\begin{pmatrix} 
a_1 &\ldots &a_{p-2} &a_{p-1}&0&0&\ldots&0&0\\
1& \ldots & 1 & 1&1&1& \ldots & 1 & 1\\
0&\ldots&0&0&0& b_1 &\ldots &b_{q-2}&b_{q-1}\\
\end{pmatrix}.
\]
Then, setting $R:=k[x_1,\ldots,x_{p-1},z,y_2,\ldots,y_q]$, the ideal $I_C$ splits as $I_C=I_A'+I_B'$ with $I_A'=I_A\vert_{x_p=z}\cdot R$ and $I_B'=I_B\vert_{y_1=z}\cdot R$.

\subsection{An application to non-homogeneous toric ideals}

Let $A$ be an $n\times p$ matrix over $\N$ whose columns generate an affine semigroup  $\sA$  in $\N^n$.  Consider the {\it homogenization} $A^H$ of $A$ by adding to $A$ a row of $1's$.  Note that if $I_A$ is homogeneous in the standard grading, then $I_A = I_{A^H}$. 

\smallskip
The ideal $I_{A^H} \subset I_A$ and it is a homogeneous prime ideal. Moreover, $\height {I_{A^H}}= \height {I_A}-1$ because $\rk{A^H}=\rk{A}+1$, unless $I_A$ is homogeneous in which case $I_A= I_{A^H}$. Thus, $I_{A^H}$ is the largest homogenous ideal contained in $I_A$. This concept has already appeared and utilized in  Zariski-Samuel \cite{Z-S}. 

\begin{definition}{\rm
Given an polynomial ideal $I$, the ideal generated by all the homogeneous elements in $I$ is called the {\it homogeneous sift} of $I$ and denoted by $I^\star$. It is the largest homogeneous ideal contained in $I$.
}\end{definition}

In \cite{Z-S}, they prove that if $I$ is prime (or primary), then so is $I^{\star}$.  In our situation, we directly get that $(I_A)^{\star}$ is prime for a toric ideal $I_A$ since $(I_A)^\star=I_{A^H}$ as shown before the definition. Indeed, $(I_A)^{\star}$ is toric.

\begin{example}\label{ex:arithSeq}{\rm 
If we consider the numerical semigroup $\sA$ generated by an arithmetic sequence as in  \cite{ja13}, i.e., 
$A=\begin{pmatrix}m_0&m_0+d&\ldots&m_0+nd\end{pmatrix}$ with $\gcd (m_0,d)=1$, then
\[A^H=\begin{pmatrix}m_0&m_0+d&\ldots&m_0+nd\\1&1&\ldots&1\end{pmatrix}\sim
\begin{pmatrix}0&d&\ldots&nd\\1&1&\ldots&1\end{pmatrix}\sim 
\begin{pmatrix}0&1&\ldots&n\\1&1&\ldots&1\end{pmatrix}.
\]
Thus, the homogeneous sift of $I_A$ is $(I_A)^\star=I_2(M)$, the ideal of the $2\times 2$ minors of the matrix $M=\begin{pmatrix}x_0&\ldots&x_{n-1}\\x_1&\ldots&x_n\end{pmatrix}$, the so-called rational normal scroll. In this case, one knows that $I_A=(I_A)^\star+I_2(B)$ for some $2\times t$ matrix $B$ defined in \cite{ja13}.
}\end{example}

\smallskip
Now given $A=\begin{pmatrix}\a_1 & \ldots& \a_p\end{pmatrix}$ and $B=\begin{pmatrix}\b_1& \ldots& \b_q\end{pmatrix}$ two $n\times p$ and $m\times q$ matrices over the natural numbers that are not homogeneous, one may not be able to glue them, but one can glue $A^H$ and $B^H$ as in Theorem \ref{thm:A+B=C} to get a matrix $C$ such that $I_C$ splits.
%
%
Let $I_{1\times t}$ denote the $1\times t$ matrix all of whose entries are $1$, and let $(\a_p)$, resp. $(\b_1)$, be the $p\times 1$, resp. $q\times 1$, submatrix of $A$, resp. $B$, formed by its last, resp. first, column.  

\smallskip
We have
\begin{equation}\label{eq:gluingSifts}
    C =  \begin{pmatrix} 
A &  (\a_p) \times 1_{1\times q-1} \\
1_{1\times p}&1_{1\times q-1}\\
(\b_1) \times 1_{1\times p-1}&   B\\
\end{pmatrix}\,.
\end{equation}
Note that the matrix $C$ is an $(m+n+1)\times (p+q-1)$ matrix. Thus, in above block representation of $C$, $(\b_1) \times 1_{1\times p-1}$ is a $m \times (p-1)$ matrix. 

\smallskip
As in Theorem \ref{thm:A+B=C}, if $I_A' \subset k[x_1, \ldots, x_{p-1},z]$ and $I_B' \subset k[z, y_2, \ldots, y_q]$, we will have $I_C=I_{A^H}+I_{B^H}$ is a homogeneous toric ideal of height $\height{I_A}+\height{I_B}-2$, and it is contained in $I_A'+I_B'$ that may not be a prime ideal, but the homogenous prime (i.e., toric) ideal $I_{A^H}+I_{B^H}$ is contained in $I_A'+I_B'$.

\begin{example}{\rm
Consider the matrices $A= \begin{pmatrix} 
1&1&0&3&2 \\
2&1&1&0&2
\end{pmatrix}$ and $B=\begin{pmatrix} 5&12&13\end{pmatrix}$. Then, 
$A^H= \begin{pmatrix} 
1&1&0&3&2 \\
2&1&1&0&2\\
1&1&1&1&1
\end{pmatrix}$ and $B^H=\begin{pmatrix} 5&12&13\\ 1&1&1\end{pmatrix}$. 
One can check using, e.g., \cite{Sing} that $I_A$ and $I_B$ are nonhomegeneous toric ideals both minimally generated by 3 binomials, while $(I_A)^\star=I_{A^H}=\langle x_1x_2-x_3x_5, x_1^2x_3x_4-x_2^3x_5,  x_1^3x_4-x_2^2x_5^2, x_2^4-x_1x_3^2x_4\rangle$ and $(I_B)^\star=I_{B^H}=\langle y_2^8-y_1y_3^7\rangle$.
For $C=
\begin{pmatrix}
1&1&0&3&2&2&2 \\
2&1&1&0&2&2&2\\
1&1&1&1&1&1&1\\
5&5&5&5&5&12&13
\end{pmatrix}
$ as in \eqref{eq:gluingSifts},
$I_C=I_{A^H}\vert_{x_5=z}+I_{B^H}\vert_{y_1=z}=\langle 
x_1x_2-x_3z, x_1^2x_3x_4-x_2^3z, x_1^3x_4-x_2^2z^2, x_2^4-x_1x_3^2x_4, y_2^8-zy_3^7
\rangle\subset R=k[x_1,\ldots,x4,z,y_2,y_3]$. Note that $I_C=(I_D)^\star$ where $D$ is the $3\times 7$ matrix obtained by removing the row of 1s in $C$. 
In this example, $I_D$ is minimally generated by 10 nonhomogeneous binomials.
}
\end{example}

\section{Application to Gluing and splitting of Graphs}\label{sec:graphs}

Let $G=(V(G), E(G))$ be  a graph with vertex set $V(G)=\{t_1,\ldots,t_{n_G}\}$ and edge set $E(G)=\{x_1,\ldots,x_{p_G}\}$.  The incidence matrix of $G$ is the $n_G\times p_G$ matrix $A(G)$ such that $A(G)_{ij}=1$ if $t_i\in x_j$ and $0$ otherwise. Note that each column of $A(G)$ has exactly 2 non-zero entries that are both equal to 1. The semigroup associated to the graph $G$ is the subsemigroup $\sAg$ of $\N^{n_G}$  generated by the columns of $A(G)$ and one can consider the semigroup ring of $\sAg$, $\rAg$.   Since each row adds up to $2$, the matrix $A(G)$ and the semigroup ring $k[A(G)]$ are homogenous.  

\smallskip
One can define a notion of gluing and splitting on graphs in general. Two graphs can be glued by identifying the vertices and edges of isomorphic subgraphs as in \cite[Construction 4.1]{FHKV}. In particular, one can always glue two graphs along an edge. This gluing and splitting along an edge is defined below. 

\begin{definition} 
{\rm Consider two disjoint connected graphs $G_1$ and $G_2$ and choose two  vertices $a,b$ in $G_1$ connected by an edge $e_1$, and two vertices $c,d$ in $G_2$ connected by an edge $e_2$. The graph $G$ obtained by identifying $a$ with $c$ and $b$ with $d$ (and by identifying the corresponding edges $e_1$ and $e_2$) is called a graph obtained by {\it gluing $G_1$ and $G_2$ along one edge} as in Figure \ref{fig:glueSquares}. Note that there are two different ways of gluing two graphs along the same edge, and that the graphs $G_1$ and $G_2$ become induced subgraphs of $G$ after gluing.

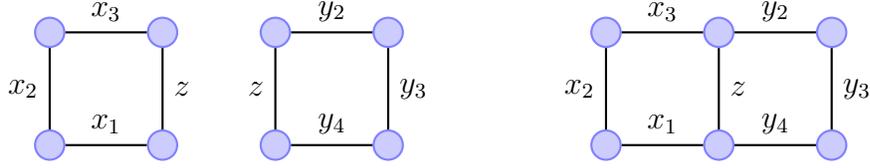
\begin{figure}
\centering
\begin{tikzpicture}[node distance={15mm}, thick, main/.style = {circle,draw=blue!50,fill=blue!20}]
\node[main] (1) {}; 
\node[main] (2) [right of=1] {};
\node[main] (3) [below of=1] {};
\node[main] (4) [right of=3] {};
\draw (3) -- node[midway, above] {$x_1$} (4);
\draw (3) -- node[midway, left] {$x_2$} (1);
\draw (1) -- node[midway, above] {$x_3$} (2);
\draw (2) -- node[midway, right] {$z$} (4);
\node[main] (5) [right of=2] {}; 
\node[main] (6) [right of=5] {};
\node[main] (7) [below of=5] {};
\node[main] (8) [right of=7] {};
\draw (7) -- node[midway, above] {$y_4$} (8);
\draw (7) -- node[midway, left] {$z$} (5);
\draw (5) -- node[midway, above] {$y_2$} (6);
\draw (6) -- node[midway, right] {$y_3$} (8);
\end{tikzpicture} 
\hskip 40pt
\begin{tikzpicture}[node distance={15mm}, thick, main/.style = {circle,draw=blue!50,fill=blue!20}]
\node[main] (1) {}; 
\node[main] (2) [right of=1] {};
\node[main] (3) [below of=1] {};
\node[main] (4) [right of=3] {};
\draw (3) -- node[midway, above] {$x_1$} (4);
\draw (3) -- node[midway, left] {$x_2$} (1);
\draw (1) -- node[midway, above] {$x_3$} (2);
\draw (2) -- node[midway, right] {$z$} (4);
\node[main] (6) [right of=2] {};
\node[main] (8) [right of=4] {};
\draw (4) -- node[midway, above] {$y_4$} (8);
\draw (2) -- node[midway, above] {$y_2$} (6);
\draw (6) -- node[midway, right] {$y_3$} (8);
\end{tikzpicture} 
\caption{The graph $G$ on the right is obtained by gluing two squares $G_1$ and $G_2$ along an edge
} \label{fig:glueSquares}
\end{figure}

Conversely, if a graph $G$ has one edge $e$ such that $G$ can be obtained by gluing two disjoint graphs $G_1$ and $G_2$ along one edge that becomes $e$ in $G$, we will say that $G$ {\it splits into $G_1$ and $G_2$ along the edge $e$} as in Figure \ref{fig:splitSquares}.

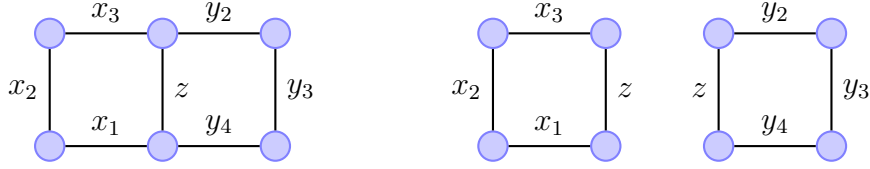
\begin{figure}
\centering
\begin{tikzpicture}[node distance={15mm}, thick, main/.style = {circle,draw=blue!50,fill=blue!20}]
\node[main] (1) {}; 
\node[main] (2) [right of=1] {};
\node[main] (3) [below of=1] {};
\node[main] (4) [right of=3] {};
\draw (3) -- node[midway, above] {$x_1$} (4);
\draw (3) -- node[midway, left] {$x_2$} (1);
\draw (1) -- node[midway, above] {$x_3$} (2);
\draw (2) -- node[midway, right] {$z$} (4);
\node[main] (6) [right of=2] {};
\node[main] (8) [right of=4] {};
\draw (4) -- node[midway, above] {$y_4$} (8);
\draw (2) -- node[midway, above] {$y_2$} (6);
\draw (6) -- node[midway, right] {$y_3$} (8);
\end{tikzpicture} 
\hskip 40pt
\begin{tikzpicture}[node distance={15mm}, thick, main/.style = {circle,draw=blue!50,fill=blue!20}]
\node[main] (1) {}; 
\node[main] (2) [right of=1] {};
\node[main] (3) [below of=1] {};
\node[main] (4) [right of=3] {};
\draw (3) -- node[midway, above] {$x_1$} (4);
\draw (3) -- node[midway, left] {$x_2$} (1);
\draw (1) -- node[midway, above] {$x_3$} (2);
\draw (2) -- node[midway, right] {$z$} (4);
\node[main] (5) [right of=2] {}; 
\node[main] (6) [right of=5] {};
\node[main] (7) [below of=5] {};
\node[main] (8) [right of=7] {};
\draw (7) -- node[midway, above] {$y_4$} (8);
\draw (7) -- node[midway, left] {$z$} (5);
\draw (5) -- node[midway, above] {$y_2$} (6);
\draw (6) -- node[midway, right] {$y_3$} (8);
\end{tikzpicture} 
\caption{The graph $G$ on the left splits into two squares} \label{fig:splitSquares}
\end{figure}
}\end{definition}

If $G$ is a gluing of $G_1$ and $G_2$ along one edge, then $n_G=n_{G_1}+n_{G_2}-2$ and $p_G=p_{G_1}+p_{G_2}-1$
The three incidence matrices $A(G_1)$, $A(G_2)$ and $A(G)$ are related as in the lemma below.

\begin{lemma}\label{lem:MatrixGluingGraph}
    If $G$ is a gluing of two connected graphs $G_1$ and $G_2$ along one edge then, up to permutation of rows and columns, the incidence matrices of $A(G_1)$, $A(G_2)$ and $A(G)$ are as follows: 
\[
A(G_1)=\left(
\begin{array}{cc}
      & 0 \\
      &\vdots\\
A(G_1')&0\\
      & 1 \\
      & 1
\end{array}\right),\quad 
A(G_2)=\left(
\begin{array}{cc}
1     &  \\
1     &\\
0     &A(G_2')\\
\vdots& \\
0      & 
\end{array}\right),\quad
A(G)=\left(
\begin{array}{ccc}
     & 0 &\\
 &\vdots&0\\
 A(G_1')  &0&\\
     & 1 & \\
     & 1 & \\
    &0 & A(G_2') \\
 0  &\vdots&\\
    &0& 
\end{array}\right).
\]
\end{lemma}

\begin{proof}
First note that permutation of rows, respectively columns, of the incidence matrix 
correspond to renaming vertices, respectively edges, on the graph.
Hence, the order of the columns and rows of the incidence matrix of a graph is not relevant and one can locate the column corresponding to the gluing edge $e$ as the last, respectively first, column of $A(G_1)$, respectively $A(G_2)$. Moreover, the last, respectively first, two rows of $A(G_1)$, respectively $A(G_2)$, correspond to the two vertices of the gluing edge $e$ (in any order). Then, the matrices $A(G_1)$ and $A(G_2)$ show as in the theorem 
where $A(G_i')$ is the incidence matrix of the graph $G_i'$ with $E(G_i')=E(G_i)\setminus\{e\}$ and $V(G_i')=V(G_i)$ for $i=1,2$.

\smallskip
The glued graph $G$ has $n_G=n_{G_1}+n_{G_2}-2$ vertices and $p_G=p_{G_1}+p_{G_2}-1$ edges and its incidence matrix $A(G)$ is the $n_G\times p_G$ matrix given in the theorem.
\end{proof}

By \cite[Lem. 10.2.6]{villarreal}, the rank of the incidence matrix of a connected graph $G$, and hence the dimension of the semigroup ring $\rAg]$ is given by 
\begin{equation}\label{eq:dim}
\dim{\rAg}=\rk{A(G)}=\left\{\begin{array}{cl}
  n_G-1   & \hbox{if $G$ is bipartite},\\
  n_G   &  \hbox{otherwise}.
\end{array}
\right.
\end{equation}

Note that the graph $G$ in Lemma \ref{lem:MatrixGluingGraph} is bipartite if and only if both $G_1$ and $G_2$ are. This is easily seen directly on the graph because a graph is bipartite if and only if it has no odd cycles.  
%
Thus, by \eqref{eq:dim}, one gets that 
\[
\dim{\rAg}=\left\{\begin{array}{cl}
  n_G-1   &  \hbox{if both $G_1$ and $G_2$ are bipartite},\\
  n_G   & \hbox{otherwise}.
\end{array}
\right.
\]

\begin{remark}\label{rem:bipartite}{\rm 
Using the previous observation, if $G_1$ and $G_2$ are two connected graphs and $G$ is a gluing of $G_1$ and $G_2$ along an edge:
\begin{itemize}
    \item if $G_1$ and $G_2$ are both bipartite, then $\rk{A(G)}=n_{G_1}+n_{G_2}-3=\rk{A(G_1)}+\rk{A(G_2)}-1$,
    \item if $G_1$ and $G_2$ are both non-bipartite, then $\rk{A(G)}=n_{G_1}+n_{G_2}-2=\rk{A(G_1)}+\rk{A(G_2)}-2$, and
    \item if one of the two graphs is bipartite and the other is not, then $\rk{A(G)}=n_{G_1}+n_{G_2}-2=\rk{A(G_1)}+\rk{A(G_2)}-1$.
\end{itemize}
In particular, $\rk{A(G)}=\rk{A(G_1)}+\rk{A(G_2)}-1$ if and only if at least one of the two graphs is bipartite.
}\end{remark}

\smallskip
As an application of Theorem \ref{thm:A+B=C} and Remark \ref{rem:bipartite}, we recover and extend an interesting result in \cite{FHKV}.
Suppose that $G_1$ and $G_2$ is a splitting of $G$ along an edge.  In \cite[Cor. 4.8]{FHKV}, they prove that if at least one of the two graphs is bipartite, then the toric ideal $I_G$ splits into $I_G=I_{G_1}+I_{G_2}$.   We prove the converse as well, and the fact that a splitting of graph automatically creates a toric splitting by going to hypergraphs if needed.

\begin{theorem}\label{thm:gluingGraphs}
\begin{enumerate}
    \item\label{thm:item:1bipartite}
Let $G$ be a connected graph, and suppose that $G_1$ and $G_2$ is a splitting of $G$ along an edge. The toric ideal $I_G$ splits into $I_G=I_{G_1}+I_{G_2}$ if and only if least one of the two graphs is bipartite.
\item\label{thm:item:nonebipartite}
If neither $G_1$ nor $G_2$ are bipartite, there is indeed a $3$-uniform hypergraph $\tilde{G}$ such that the toric ideal $I_{\tilde{G}}$ splits as $I_{G_1}+I_{G_2}$.
    \item\label{thm:item:hypergraph}
If $G_1$ and $G_2$ are two arbitrary disjoint connected graphs, bipartite or not, there exists a $3$-uniform hypergraph $\tilde{G}$ such that the toric ideal $I_{\tilde{G}}$ splits as $I_{G_1}+I_{G_2}$.
\end{enumerate}
\end{theorem}

\begin{proof}
Remark \ref{rem:bipartite} implies that the necessary rank condition in \cite{Res22} for gluing is satisfied if and only if at least one of the graphs $G_1$ and $G_2$ is bipartite.  So, $A(G)$ in Lemma \ref{lem:MatrixGluingGraph} produces a splitting on the toric ideals
if and only if at least one of them is bipartite. Thus by Theorem \ref {thm:gluing} and applying the same argument as in the proof of Theorem \ref{thm:A+B=C}, we know that $I_G$ has  a toric splitting as $I_{G_1} + I_{G_2}$.  

Further, if neither of the graphs are bipartite, since the incidence matrices are homogenous, we can indeed glue them as is Theorem \ref{thm:A+B=C} and the resulting matrix $C$ is the incidence matrix of a $3$-uniform hypergraph $\tilde{G}$ with number of vertices $n_{\tilde{G}} = n_{G_1}+n_{G+2}-1$ and number of edges (triangles) equal to $p_{\tilde{G}} = p_{G_1}+p_{G_2} -1$. 
\end{proof}

\begin{example}{\rm
\begin{enumerate}
    \item\label{item:gluingSquares} 
    If $G$ is the graph obtained by gluing two squares $G_1$ and $G_2$ along an edge as shown in Figure \ref{fig:glueSquares}, then $I_{G_1}=\langle x_1x_3-x_2z\rangle$, $I_{G_2}=\langle zy_3-y_2y_4\rangle$,
    and $I_{G}$ splits as $I_{G}=I_{G_1}+I_{G_2}$ as expected by Theorem \ref{thm:gluingGraphs} \eqref{thm:item:1bipartite}.
    \item If $G_1$ and $G_2$ are two bow ties (both non-bipartite graphs) then, by Theorem \ref{thm:gluingGraphs} \eqref{thm:item:1bipartite}, if we glue them along an edge to obtain a graph $G$ as in Figure \ref{fig:glueBowties}, $I_G$ does not split as $I_{G}=I_{G_1}+I_{G_2}$. Indeed, $I_{G_1}=\langle x_1x_3x_5-x_2x_4z\rangle$, $I_{G_2}=\langle zy_3y_5-y_2y_4y_6\rangle$, and the ideal $I_G$ is generated by 5 binomials.
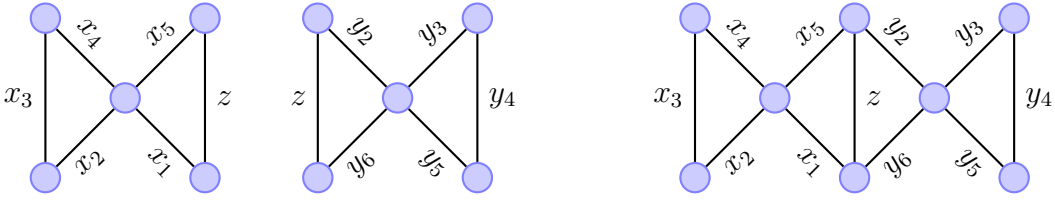
\begin{figure}
\centering
\begin{tikzpicture}[node distance={15mm}, thick, main/.style = {circle,draw=blue!50,fill=blue!20}]
\node[main] (1) {}; 
\node[main] (9) [below right of=1] {};
\node[main] (2) [above right of=9] {};
\node[main] (3) [below left of=9] {};
\node[main] (4) [below right of=9] {};
\draw (9) -- node[midway, below left, pos=1, sloped] {$x_1$} (4);
\draw (9) -- node[midway, below right, pos=1, sloped] {$x_2$} (3);
\draw (3) -- node[midway, left] {$x_3$} (1);
\draw (9) -- node[midway, above right, pos=1, sloped] {$x_4$} (1);
\draw (9) -- node[midway, above left, pos=1, sloped] {$x_5$} (2);
\draw (2) -- node[midway, right] {$z$} (4);
\node[main] (5) [right of=2] {}; 
\node[main] (10) [below right of=5] {};
\node[main] (6) [above right of=10] {};
\node[main] (7) [below left of=10] {};
\node[main] (8) [below right of=10] {};
\draw (10) -- node[midway, below left, pos=1, sloped] {$y_5$} (8);
\draw (10) -- node[midway, below right, pos=1, sloped] {$y_6$} (7);
\draw (7) -- node[midway, left] {$z$} (5);
\draw (10) -- node[midway, above right, pos=1, sloped] {$y_2$} (5);
\draw (10) -- node[midway, above left, pos=1, sloped] {$y_3$} (6);
\draw (6) -- node[midway, right] {$y_4$} (8);
\end{tikzpicture} 
\hskip 40pt
\begin{tikzpicture}[node distance={15mm}, thick, main/.style = {circle,draw=blue!50,fill=blue!20}]
\node[main] (1) {}; 
\node[main] (9) [below right of=1] {};
\node[main] (2) [above right of=9] {};
\node[main] (3) [below left of=9] {};
\node[main] (4) [below right of=9] {};
\draw (9) -- node[midway, below left, pos=1, sloped] {$x_1$} (4);
\draw (9) -- node[midway, below right, pos=1, sloped] {$x_2$} (3);
\draw (3) -- node[midway, left] {$x_3$} (1);
\draw (9) -- node[midway, above right, pos=1, sloped] {$x_4$} (1);
\draw (9) -- node[midway, above left, pos=1, sloped] {$x_5$} (2);
\draw (2) -- node[midway, right] {$z$} (4);
\node[main] (10) [below right of=2] {};
\node[main] (6) [above right of=10] {};
\node[main] (8) [below right of=10] {};
\draw (10) -- node[midway, below left, pos=1, sloped] {$y_5$} (8);
\draw (10) -- node[midway, below right, pos=1, sloped] {$y_6$} (4);
\draw (10) -- node[midway, above right, pos=1, sloped] {$y_2$} (2);
\draw (10) -- node[midway, above left, pos=1, sloped] {$y_3$} (6);
\draw (6) -- node[midway, right] {$y_4$} (8);
\end{tikzpicture} 
\caption{The graph $G$ on the right is obtained by gluing two bow ties $G_1$ and $G_2$ along an edge
} \label{fig:glueBowties}
\end{figure}
    \item Now by Theorem \ref{thm:gluingGraphs} \eqref{thm:item:nonebipartite}, if one glues the incidence matrices of two bow ties as in Theorem \ref{thm:A+B=C}, one obtains the incidence matrix of a 3-regular hypergraph $\tilde{G}$ such that $I_{\tilde{G}}$ splits as $I_{\tilde{G}}=I_{G_1}+I_{G_2}$.
    \item Finally, by Theorem \ref{thm:gluingGraphs} \eqref{thm:item:hypergraph}, going back to the case of two squares $G_1$ and $G_2$ in \eqref{item:gluingSquares}, one can also use Theorem \ref{thm:A+B=C} to glue them into a 3-regular hypergraph $\tilde{G}$ such that $I_{\tilde{G}}$ splits as $I_{\tilde{G}}=I_{G_1}+I_{G_2}=I_G$ for the graph $G$ defined in \eqref{item:gluingSquares}.
\end{enumerate}
}\end{example}
    
\section{Further Examples}\label{sec:examples}

\subsection{Selfgluing}

We can sometimes glue a semigroup $\sA$ with itself in order to obtain a new semigroup $\sC$ whose defining ideal will be the sum of two copies of the original toric ideal $I_A$, plus a gluing binomial. When this can be done in such a way that the gluing binomial is of the form $x_i-y_j$ like in Theorem \ref{thm:gluing} when $\sA$ is homogeneous, then we get a toric ideal that splits as the sum of two copies of the same ideal as we will see later in Examples \ref{ex:selfglueCM} and \ref{ex:selfglueNoneCM}.
Let us give now an example of a numerical semigroup that can be also selfglued in such a way that the resulting ideal splits.

\begin{example}\label{ex:numericalSelfglued}
{\rm
Two numerical semigroups can always be glued so any numerical semigroup can be selfglued but there are many ways to do it. Consider the numerical semigroup $\sA=\langle 5,12,13,16\rangle$. The defining ideal of $k[A]$ is minimally generated by 6 nonhomogeneous binomials and consider two copies of this ideal in disjoint sets of variables, $I_A\subset k[x_1,\ldots,x_4]$ and $I_B\subset k[y_1,\ldots,y_4]$. Choosing, for example, $k_1=17=5+12$ and $k_2=18=5+13$, one gets that for $C=k_1\cdot A\cup k_2\cdot A=\langle 85,204,221,272,90,216,234,288\rangle$, $\sC$ is a selfgluing of $\sA$ and $I_C=I_A+I_B+\langle x_1x_3-y_1y_2\rangle$. If one now chooses $k_1=5$ and $k_2=16$, the gluing binomial becomes $x_4-y_1$ and hence this provides a splitting: the toric ideal $I_D\subset R:=k[x_1,x_2,x_3,z,y_2,y_3,y_4]$ defined by the numerical semigroup $\sD=\langle 25,60,65,80,192,208,256\rangle$ splits as the sum of two copies of the same toric ideal, $I_D=I_A\vert_{x_4=z}\cdot R+I_B\vert_{y_1=z}\cdot R$. 
}\end{example}

\subsection{Constructing Gorenstein, Cohen-Macaulay or non Cohen-Macaulay toric varieties of higher dimensions}

We saw by theorem \ref{thm:gluing} that we can embed nondegenerate homogeneous semigroups in higher dimensional spaces where they are degenerate and glue them there. Moreover, we can do that in such a way that we get an ideal that splits as the sum of our two original ideals.  This can be used to build Gorenstein, Cohen-Macaulay or non Cohen-Macaulay surfaces (or toric varieties of higher dimension in general). This is because that a gluing of two semigroups is Cohen-Macualay (resp. Gorenstein) if and only if the two semigroups are Cohen-Macaulay (resp. Gorenstein) as shown in \cite[Thm. 1.5]{Res22}.  Let's give two examples, the first one obtained by selfgluing a Cohen-Macaulay curve in $\PP_k^3$ with itself to get a Cohen-Macaulay surface in $\PP_k^6$, and the other by selfgluing a non Cohen-Macaulay curve in $\PP_k^3$ with itself to get a non Cohen-Macaulay surface in $\PP_k^6$.

\begin{example}\label{ex:selfglueCM}
{\rm
The monomial curve in $\PP_k^3$ defined by the semigroup $\sA$ generated by the columns of the matrix
$A=\begin{pmatrix} 5&4&3&0\\ 0&1&2&5\end{pmatrix}$ is known to be Cohen-Macaulay. It is indeed of Hilbert-Burch and 
$I_A=\langle x_1x_3-x_2^2, x_1x_2x_4-x_3^3, x_1^2x_4-x_2x_3^2\rangle$.
Then, for 
$\tilde{C}=\begin{pmatrix} 5&4&3&0&0&0&0&0\\ 0&1&2&5&5&4&3&0\\ 0&0&0&0&0&1&2&5\end{pmatrix}$,
the semigroup generated by the columns of $\tilde{C}$ defines a projective surface  in $\PP_k^7$ that is Cohen-Macaulay because it is the gluing of two Cohen Macaulay curves.
Considering now the submatrix $C$ of $\tilde{C}$ obtained by removing the repeated central column, 
one gets a toric ideal $I_C$ defining an arithmetically Cohen-Macaulay monomial surface in $\PP_k^6$. The ideal $I_C$ splits as the sum of two copies of the same ideal in different sets of variables, one in $k[x_1,x_2,x_3,z]$ ($I_A\vert_{x_4=z}$) and one in $k[z,y_2,y_3,y_4]$.
}\end{example}

\begin{example}\label{ex:selfglueNoneCM}
{\rm
Doing the same with the matrix $A=\begin{pmatrix} 5&4&1&0\\ 0&1&4&5\end{pmatrix}$, one gets that $I_A=\langle x_1x_4-x_2x_3, x_2x_4^3-x_3^4, x_1x_3^3-x_2^2x_4^2, x_1^2x_3^2-x_2^3x_4, x_1^3x_3-x_2^4\rangle$ defines a non arithmetically Cohen-Macaulay curve in $\PP_k^3$. Thus, the matrix $C=\begin{pmatrix} 5&4&1&0&0&0&0\\ 0&1&4&5&4&1&0\\ 0&0&0&0&1&4&5\end{pmatrix}$ defines a non arithmetically Cohen-Macaulay surface in $\PP_k^6$ whose defining ideal $I_C$ splits as the sum of two copies of $I_A$.
}\end{example}

\subsection{Iterated gluing and splitting}

The gluing construction can be iterated. This can always be done in the numerical and in the homogeneous cases. Again, when the gluing binomial is linear (and homogeneous), one gets a splitting: iterating this construction provides a splitting of a toric ideal in more than 2 parts.

\begin{example}\label{ex:iteratedNumerical}
Consider $A_1 = \begin{pmatrix} 5&8&11\end{pmatrix}$, 
$A_2= \begin{pmatrix} 7&10&12\end{pmatrix}$ and $A_3 = \begin{pmatrix} 6&11&14\end{pmatrix}$, and the associated toric ideals $I_{A_1}\subset k[x_1,x_2,x_3]$, $I_{A_2}\subset k[y_1,y_2,y_3]$ and $I_{A_3}\subset k[z_1,z_2,z_3]$, all of them minimally generated by 3 binomials. Choosing $k_1=17\in\langle A_2\rangle$ and $k_2=13\in\langle A_1\rangle$, one gets that for $C=k_1\cdot A_1\cup k_2\cdot A_2$, $\sC$ is a gluing of $\langle A_1\rangle$ and $\langle A_2\rangle$ and $I_C=I_{A_1}+I_{A_2}+\langle x_1x_2-y_1y_2\rangle$. One can now glue $\sC$ with $\langle A_3\rangle$ choosing, for instance, $k_1'=17\in \langle A_3\rangle$ and $k_3=85+91\in\sC$ and we finally get that for $D=(17)^2\cdot A_1 \cup (17\times 13)\cdot A_2 \cup (85+91)\cdot A_3$, i.e., 
\[D=\begin{pmatrix} 1445&2312&3179&1547&2210&2652&1056&1936&2464\end{pmatrix},\]
the numerical semigroup $\sD$ is a gluing of $\langle A_1\rangle$, $\langle A_2\rangle$ and $\langle A_3\rangle$ and $I_D=I_{A_1}+I_{A_2}+I_{A_3}+\langle x_1x_2-y_1y_2\rangle+\langle x_1y_1-z_1z_2\rangle$.
\end{example}

\begin{example}{\rm
Further, as in Example \ref{ex:numericalSelfglued}, one can also choose  $k_1$, $k_2$, $k'_1$ and $k_3$ and produce a splitting. For the same $A_1 = \begin{pmatrix} 5&8&11\end{pmatrix}$, 
$A_2= \begin{pmatrix} 7&10&12\end{pmatrix}$ and $A_3 = \begin{pmatrix} 6&11&14\end{pmatrix}$ as in Example  \ref{ex:iteratedNumerical}, if we let $k_1=7$, $k_2=11$, $k'_1=6$ and $k_3=k_2\times 7=77$, then for   
$E=\begin{pmatrix} 
210&336&462&660&792&847&1078
\end{pmatrix}$,
the toric ideal $I_E$ associated to the numerical semigroup $\langle E\rangle$ splits as the sum of three toric ideals,
\[I_E=I_{A_1}\vert_{x_3=z}+I_{A_2}\vert_{y_1=z}+I_{A_3}\vert_{z_1=z}\subset k[x_1,x_2,z,y_2,y_3,z_2,z_3]\,.\]
}\end{example}


%
%


\begin{thebibliography}{99}



\bibitem{Sing}{
 W. Decker, G.-M. Greuel, G. Pfister, H. Sch{\"o}nemann, 
{\sc Singular} {4-2-1} --- {A} computer algebra system for polynomial computations.
Available at {\tt http://www.singular.uni-kl.de} (2021).}

\bibitem{De}{
C. Delorme,
Sous-mono\"{i}des d'intersection compl\`{e}te de $N$,
 {\em Ann. Sci. \'{E}cole Norm. Sup.} (4) {\bf 9} (1976), 145--154.}

 \bibitem{FHKV}{
G. Favacchio, J. Hofscheier, G. Keiper, A. Van Tuyl,
Splitting of toric ideals,
{\em J. Algebra} {\bf 574} (2021), 409--433.}

\bibitem{FMS}{
K. G. Fischer, W. Morris, J. Shapiro, 
Affine semigroup rings that are complete intersections,
{\em Proc. Amer. Math. Soc.} {\bf 125} (1997), 3137--3145.}

\bibitem{ja13}{
P. Gimenez, I. Sengupta, H. Srinivasan,
Minimal graded free resolutions for monomial curves defined by arithmetic sequences,
{\em J. Algebra} {\bf 388} (2013), 294--310.}

\bibitem{jpaa19}{
P. Gimenez, H. Srinivasan,
The structure of the minimal free resolution of semigroup rings obtained by gluing,
{\em J. Pure Appl. Algebra} {\bf 223} (2019), 1411--1426.}

\bibitem{sforum20}{
P. Gimenez, H. Srinivasan,
Gluing semigroups: when and how. 
{\em Semigroup Forum} {\bf 101} (2020), 603-618.}

\bibitem{Res22}{
P. Gimenez, H. Srinivasan,
On gluing semigroups in $\N^n$ and the consequences.
{\em Res. Math. Sci.} {\bf 9} (2022), Paper No. 23, 14 pp}

\bibitem{He}{
J. Herzog,
Generators and relations of abelian semigroups and semigroup rings,
{\em Manuscripta Math.} {\bf 3} (1970), 175-193.
}

\bibitem{rosales}{
J. C. Rosales,
On presentations of subsemigroups of $\N^n$,
{\em Semigroup Forum} {\bf 55} (1997), 152--159.}

\bibitem{RG}{
J. C. Rosales, P. A. García-Sánchez,
On complete intersection affine semigroups,
{\em Comm. Algebra} {\bf 23} (1995), 5395--5412.}

\bibitem{sturm} {
B. Sturmfels, {\em Gr\"obner bases and convex polytopes},
University Lecture Series {\bf 8}, Amer. Math. Soc., 1996.
}

\bibitem{villarreal}{
R. H. Villarreal, {\em Monomial algebras}, second edition,
Monographs and Research Notes in Mathematics, CRC Press, Boca Raton, FL, 2015.
}

\bibitem{Z-S} {O. Zariski and P. Samuel, {\em Commutative Algebra, Volume II}, first edition, Graduate Texts in Mathematics {\bf 29}, Springer, 1960}


\end{thebibliography}
\end{document}